\documentclass[]{amsart} 
\usepackage[final]{showkeys}       
\usepackage[breaklinks]{hyperref}
\usepackage{amssymb, tabularx}
\usepackage{tikz, float}
\usetikzlibrary{calc,decorations.pathreplacing,matrix,arrows,positioning}

\newcommand{\ZZ}{\mathbb{Z}}

\newcommand{\V}{\mathcal{V}}
\newcommand{\A}{\mathbb{A}}

\theoremstyle{plain} \newtheorem{thm}{Theorem}[section]

\newtheorem{lem}[thm]{Lemma}

\theoremstyle{definition} \newtheorem{defn}[thm]{Definition}
\theoremstyle{remark} 

\newcommand{\m}[1]{\mathbb{#1}}   

\numberwithin{equation}{section}  

\newcommand{\T}{\mathcal{T}}
\newcommand{\B}{\mathbb{B}}
\newcommand{\Con}{\text{Con}}
\newcommand{\Cg}{\text{Cg}}
\newcommand{\Sg}{\text{Sg}}
\newcommand{\supp}{\text{supp}}

\title{The Variety Generated by $\A(\T)$ -- Two Counterexamples}
\author{Matthew Moore}
\date{\today}
\address{
  Vanderbilt University;
  Nashville, TN 37240;
  U.S.A.}
\email{matthew.moore@vanderbilt.edu}

\begin{document} \maketitle
\begin{abstract}
We show that $\V(\A(\T))$ does not have definable principal subcongruences
or bounded Maltsev depth. When the Turing machine $\T$ halts, $\V(\A(\T))$
is an example of a finitely generated semilattice based (and hence
congruence $\wedge$-semidistributive) variety with only finitely many
subdirectly irreducible members, all finite. This is the first known example
of a variety with these properties that does not have definable principal
subcongruences or bounded Maltsev depth.
\end{abstract}

\section{Introduction}
In $1976$, Park conjectured in \cite{ParkThesis} that every finitely
generated variety with a finite residual bound is finitely based. This
problem, known as Park's Conjecture, is still open. It has, however, been
proved with additional hypotheses.  Baker's Theorem \cite{BakerFiniteBasis}
establishes Park's Conjecture for congruence distributive varieties.
McKenzie's Theorem \cite{McKenzieFinEqBasis} establishes Park's Conjecture
for congruence modular varieties. Willard's Theorem
\cite{WillardFiniteBasisTheorem} establishes Park's Conjecture for
congruence $\wedge$-semidistributive varieties. The theorems of McKenzie and
Willard are more general than Baker's, but incomparable to one another.

Many proofs of Baker's Theorem are now known (see
\cite{JonssonFinitelyBased, MakkaiFiniteBasis, BakerWangDPSC}), and some of
the recent approaches involve simplifications and new concepts that may be
applicable to a wider class of varieties. In fact, in
\cite{WillardExtendingBaker, WillardFiniteBasisProblem} Willard specifically
asks:
\begin{enumerate}
  \item if $\A$ is finite of finite type and $\V(\A)$ has finite residual
  bound and is congruence $\wedge$-semidistributive, is it true that
  $\V(\A)$ has definable principal subcongruences? (See Definition
  \ref{defn:DPSC})

  \item if $\V$ is a congruence $\wedge$-semidistributive variety in a
  finite language and has finite residual bound, is it true that $\V$ has
  bounded Maltsev depth? (See Definition \ref{defn:bounded_Maltsev_depth})
\end{enumerate}
This paper answers both of these questions in the negative.

We examine the variety generated by McKenzie's $\A(\T)$ algebra, which
McKenzie uses in \cite{McKenzieResidualBoundNotComp} to prove that the
property of having a finite residual bound(=a finite bound on the size of
subdirectly irreducible algebras) is undecidable, and which Willard
\cite{WillardTarskisFiniteBasis} uses to give another proof that Tarski's
Finite Basis problem is undecidable. Recent work by the author in
\cite{MooreDPSCUndec} defines an algebra $\A'(\T)$ through the addition of a
new operation to $\A(\T)$. This new operation makes it possible to prove
that $\V(\A'(\T))$ has DPSC if and only if the Turing machine $\T$ halts,
and yields a third proof that Tarski's Finite Basis Problem is undecidable.
The question of whether or not the unmodified $\A(\T)$ generates a variety
with DPSC is left unaddressed in \cite{MooreDPSCUndec}, however, and is
answered here.

The question of whether the variety generated by the modified $\A(\T)$ used
in \cite{MooreDPSCUndec} has bounded Maltsev depth is intriguing. It appears
to be the case that $\V(\A'(\T))$ does not have bounded Maltsev depth when
$\T$ does not halt, so proving that it does when $\T$ halts would show that
the property of having bounded Maltsev depth is undecidable. The
straightforward approach to proving this would seem to require a different
sort of fine analysis of polynomials of $\A'(\T)$ than that used in
\cite{MooreDPSCUndec} to prove that the property of having DPSC is
undecidable.

\section{The Algebra $\A(\T)$}
The algebra $\A(\T)$ is quite complicated, and a full understanding of its
structure is not necessary for the results in this paper. We provide a full
definition for the completeness, however.

Define a \emph{Turing machine} $\T$ to be a finite list of $5$-tuples
$(s,r,w,d,t)$, called the \emph{instructions} of the machine, and
interpreted as ``if in state $s$ and reading $r$, then write $w$, move
direction $d$, and enter state $t$.'' The set of states is finite, $r,w\in
\{0,1\}$, and $d\in \{\text{L},\text{R}\}$. A Turing machine takes as input
an infinite bidirectional tape $\tau:\ZZ\to\{0,1\}$ which has finite
support. If $\T$ stops computation on some input, then $\T$ is said to have
\emph{halted} on that input. We say that the Turing machine halts (without
specifying the input) if it halts on the empty tape $\tau(x)=0$. Enumerate
the states of $\T$ as $\{\mu_0,\ldots,\mu_n\}$, where $\mu_1$ is the initial
(starting) state, and $\mu_0$ is the halting state.

Given a Turing machine $\T$ with states $\{\mu_0,\ldots,\mu_n\}$, we
associate to $\T$ an algebra $\A(\T)$. We will now describe the algebra
$\A(\T)$. Let
\[
  U = \{1,2,\text{H}\}, 
  \qquad \qquad 
  W = \{C,D,\partial C,\partial D\}, 
  \qquad \qquad 
  A = \{0\}\cup U\cup W,
\]
\[
  V_{ir}^s = \{C_{ir}^s, D_{ir}^s, M_i^r, \partial C_{ir}^s,
    \partial D_{ir}^s, \partial M_i^r\} 
  \qquad \text{for} \qquad 
  0\leq i\leq n \text{ and } \{r,s\}\subseteq \{0,1\},
\]
\[
  V_{ir} = V_{ir}^0 \cup V_{ir}^1,
  \qquad \qquad 
  V_i = V_{i0} \cup V_{i1}, 
  \qquad \qquad 
  V = \bigcup \{ V_i \mid 0\leq i\leq n \}.
\]
The underlying set of $\A(\T)$ is $A(\T) = A\cup V$. The ``$\partial$'' is
taken to be a permutation of order $2$ with domain $V\cup W$ (e.g.
$\partial\partial C = C$), and is referred to as ``bar''. It should be
mentioned that $\partial$ is \emph{not} an operation of $\A(\T)$. We now
describe the fundamental operations of $\A(\T)$. The algebra $\A(\T)$ is a
height $1$ meet semilattice with bottom element $0$:
\newcolumntype{A}{>{\centering\arraybackslash} m{.45\linewidth}}
\newcolumntype{B}{>{\centering\arraybackslash} m{.45\linewidth}}
\begin{center} \begin{tabular}{AB}
  $\displaystyle{ x\wedge y = \begin{cases}
    x & \text{if } x = y, \\
    0 & \text{otherwise.}
  \end{cases} }$
  &
  \begin{tikzpicture}[font=\small,smooth,node distance=1em]
    \node (x1) {$x_1$};
    \node (x2) [right=of x1] {$x_2$};
    \node (dots) [right=of x2] {$\cdots$};
    \node (zero) [below=of x2] {$0$};
    \draw (x1) -- (zero) (x2)--(zero) (dots) -- (zero);
  \end{tikzpicture}
\end{tabular} \end{center}
There is a binary nonassociative ``multiplication'', defined by
\begin{align*}
  & 2\cdot D = H\cdot C = D, & & 1\cdot C = C, \\
  & 2\cdot \partial D = H\cdot \partial C = \partial D, 
  & & 1\cdot \partial C = \partial C,
\end{align*}
and $x\cdot y = 0$ otherwise. Define
\begin{align*}
  J(x,y,z) = \begin{cases}
    x          & \text{if } x = y, \\
    x \wedge z & \text{if } x = \partial y, \\
    0          & \text{otherwise}, \end{cases}
  && &&
  J'(x,y,z) = \begin{cases}
    x \wedge z & \text{if } x = y, \\
    x          & \text{if } x = \partial y, \\
    0          & \text{otherwise}. \end{cases}
\end{align*}
Define
\begin{align*}
  S_0(u,x,y,z) & = \begin{cases}
    (x\wedge y) \vee (x\wedge z) & \text{if } u\in V_0, \\
    0 & \text{otherwise},
  \end{cases} \\
  S_1(u,x,y,z) & = \begin{cases}
    (x\wedge y) \vee (x\wedge z) & \text{if } u\in \{1,2\}, \\
    0 & \text{otherwise},
  \end{cases} \\
  S_2(u,v,x,y,z) & = \begin{cases}
    (x\wedge y) \vee (x\wedge z) & \text{if } u = \partial v\in V\cup W, \\
    0 & \text{otherwise}.
  \end{cases}
\end{align*}
Define
\[
  T(w,x,y,z) = \begin{cases}
    w\cdot x & \text{if } w\cdot x = y \cdot z \text{ and } (w,x)=(y,z), \\
    \partial(w\cdot x) & \text{if } w\cdot x = y\cdot z\neq 0 \text{ and } 
      (w,x)\neq (y,z), \\
    0 & \text{otherwise}.
  \end{cases}
\]

Next, we define operations that emulate the computation of the Turing
machine. First, we define an operation that when applied to certain elements
of $A(\T)^\ZZ$ will produce something that represents a ``blank tape'':
\[
  I(x) = \begin{cases}
    C_{10}^0 & \text{if } x = 1, \\
    M_1^0 & \text{if } x = \text{H}, \\
    D_{10}^0 & \text{if } x = 2, \\
    0 & \text{otherwise}.
  \end{cases}
\]
For each instruction of $\T$ of the form $(\mu_i,r,s,\text{L},\mu_j)$ and
each $t\in \{0,1\}$ define an operation
\[
  L_{irt}(x,y,u) = \begin{cases}
    C_{jt}^{s'} & \text{if } x = y = 1 \text{ and } u = C_{ir}^{s'} 
      \text{ for some } s', \\
    M_j^t & \text{if } x = \text{H}, y = 1, \text{ and } u = C_{ir}^t, \\
    D_{jt}^s & \text{if } x = 2, y = \text{H}, \text{ and } u = M_i^r, \\
    D_{jt}^{s'} & \text{if } x = y = 2 \text{ and } u = D_{ir}^{s'} 
      \text{ for some } s', \\
    \partial v & \text{if } u\in V \text{ and } 
      L_{irt}(x,y,\partial u) = v\in V \text{ by the above lines}, \\
    0 & \text{otherwise}.
  \end{cases}
\]
Let $\mathcal{L}$ be the set of all such operations. Similarly, for each
instruction of $\T$ of the form $(\mu_i,r,s,\text{R},\mu_j)$ and each $t\in
\{0,1\}$ define an operation
\[
  R_{irt}(x,y,u) = \begin{cases}
    C_{jt}^{s'} & \text{if } x = y = 1 \text{ and } u = C_{ir}^{s'} 
      \text{ for some } s', \\
    C_{jt}^s & \text{if } x = \text{H}, y = 1, \text{ and } u = M_i^r, \\
    M_j^t & \text{if } x = 2, y = \text{H}, \text{ and } u = D_{ir}^t, \\
    D_{jt}^{s'} & \text{if } x = y = 2 \text{ and } u = D_{ir}^{s'} 
      \text{ for some } s', \\
    \partial v & \text{if } u\in V \text{ and } 
      R_{irt}(x,y,\partial u) = v\in V \text{ by the above lines}, \\
    0 & \text{otherwise}.
  \end{cases}
\]
Let $\mathcal{R}$ be the set of all such operations. When applied to certain
elements from $\A(\T)^\ZZ$, these operations simulate the computation of
the Turing machine $\T$ on different inputs. Certain elements of
$\{1,2,H\}^{\ZZ}$ serve to track the position of the Turing machine's head
when operations from $\mathcal{L}\cup \mathcal{R}$ are applied to elements
of $\A(\T)^{\ZZ}$ that encode the contents of the tape. For this reason, we
define a binary relation $\prec$ on $\{1,2,\text{H}\}$ by $x\prec y$ if and
only if $x = y = 2$, or $x = 2$ and $y = \text{H}$, or $x = y = 1$. For
$F\in \mathcal{L}\cup \mathcal{R}$ note that $F(x,y,z) = 0$ except when
$x\prec y$.  Next we define two operations for each $F\in \mathcal{L}\cup
\mathcal{R}$,
\begin{align*}
  U_F^1(x,y,z,u) & = \begin{cases}
    \partial F(x,y,u) & \text{if } x\prec z, y\neq z, F(x,y,u)\neq 0, \\
    F(x,y,u) & \text{if } x\prec z, y = z, F(x,y,u)\neq 0, \\
    0 & \text{otherwise},
  \end{cases} \\
  U_F^0(x,y,z,u) & = \begin{cases}
    \partial F(y,z,u) & \text{if } x\prec z, x\neq y, F(y,z,u)\neq 0, \\
    F(y,z,u) & \text{if } x\prec z, x = y, F(y,z,u) \neq 0, \\
    0 & \text{otherwise}.
  \end{cases}
\end{align*}

The operations on $\A(\T)$ are
\[
  \{0, \wedge, (\cdot), J, J', S_0, S_1, S_2, T, I\} \cup \mathcal{L} \cup
  \mathcal{R} \cup \{U_F^1, U_F^2 \mid F\in \mathcal{L} \cup \mathcal{R} \}.
\]
$\left<A(\T); \wedge\right>$ is a height $1$ semilattice, so there is an
order on $A(\T)$ determined by this semilattice structure: $x\leq y$ if and
only if $x\in \{0,y\}$. All the operations of $\A(\T)$ are monotone with
respect to this order. That is, if $F(x_1,\ldots,x_n)$ is any operation of
$\A(\T)$ and $a_1,b_1,\ldots,a_n,b_n\in A(\T)$ then
\[
  (a_1,\ldots,a_n)\leq (b_1,\ldots,b_n)
  \qquad \qquad \text{implies} \qquad \qquad
  F(a_1,\ldots,a_n)\leq F(b_1,\ldots,b_n).
\]
In \cite{MooreDPSCUndec}, the author extends the language of the $\A(\T)$
algebra by adding the operation
\[
  K(x,y,z) = \begin{cases}
    y & \text{if } x = \partial y, \\
    z & \text{if } x = y = \partial z, \\
    x\wedge y\wedge z & \text{otherwise},  \end{cases}
\]
(this operation is also monotone with respect to $\leq$). The resulting
algebra is denoted by $\A'(\T)$. A very fine analysis of the polynomials of
the variety generated by this algebra proves that $\V(\A'(\T))$ has DPSC if
and only if the Turing machine $\T$ halts, thereby proving that the property
of having DPSC is in general undecidable.

The purpose of all these constructions is to prove the following theorem.
\begin{thm}
The following are equivalent.
\begin{enumerate}
  \item $\T$ halts,
  \item $\V(\A(\T))$ has finite residual bound (McKenzie
    \cite{McKenzieResidualBoundNotComp}),
  \item $\V(\A(\T))$ is finitely based (Willard
    \cite{WillardTarskisFiniteBasis}),
  \item  $\V(\A'(\T))$ has definable principal subcongruences (the author
    \cite{MooreDPSCUndec}).
\end{enumerate}
\end{thm}

Congruence $\wedge$-semidistributivity is a generalization of congruence
distributivity, and turns out to be an important property for $\V(\A(\T))$.

\begin{defn} \label{defn:cong_SD}
A class $\mathcal{C}$ of algebras is said to be \emph{congruence
$\wedge$-semidistributive} if the congruence lattice of each algebra in
$\mathcal{C}$ satisfies the $\wedge$-semidistributive law:
\[
  \left[ x\wedge y = x\wedge z \right] 
    \rightarrow \left[ x\wedge y = x\wedge (y\vee z) \right].
\]
\end{defn}

Algebras $\A$ that generate congruence $\wedge$-semidistributive varieties
include those with a fundamental operation $\wedge$ such that $\left< A;
\wedge \right>$ is a semilattice. $\A(\T)$ is clearly such an algebra.

\section{$\V(\A(\T))$ Does Not Have DPSC or Bounded Maltsev Depth}

The properties of DPSC and bounded Maltsev depth are properties that each
subclass of the variety must possess. Exhibiting a subclass of the variety
that cannot have these properties therefore proves that the entire variety
cannot have these properties. We will now define such a subclass. Fix $n\geq
2$ and define elements of $\A(\T)^n$
\begin{align*}
  b_i & = (D,D,\ldots, \stackrel{i}{\hat{D}},0,\ldots,0), & 
    d_i & = (D,\ldots,D,\stackrel{i}{\hat{\partial D}},0,\ldots,0), \\
  c_i & = (0,D,\ldots,\stackrel{i}{\hat{D}},0,\ldots,0).
\end{align*}
Let $a = b_1$ and define
\[
  \B_n = \Sg^{\A(\T)^n} \left( \{ a, b_i, d_i \mid 2\leq i\leq n \} \right).
\]
Note that the only fundamental operations of $\A(\T)$ that are nonzero on
the generators (and hence on $B_n$) are $\wedge$, $J$, $J'$, and $S_2$. For
both DPSC and bounded Maltsev depth we will be examining the congruence
$\Cg^{\B_n}(a,0)$, but first we will give some useful properties of $B_n$.

\begin{lem} \label{lem:B_n_properties}
If $x\in B_n$ then
\begin{enumerate}
  \item $x(1)\in \{0,D\}$ and $x(l)\in \{0,D,\partial D\}$

  \item there is at most one $l$ such that $x(l) = \partial D$,

  \item if $x(l) = \partial D$ then $x(k) = 0$ for all $k > l$, and

  \item if $x(l) = \partial D$ then either $x = d_l$ or there is $k<l$ such
  that $x(k) = 0$.
\end{enumerate}
(We take the index of the first coordinate to be $1$.)
\end{lem}
\begin{proof}
The first part of the first item is a consequence of that fact that
$\pi_1(\{a,b_i,d_i\mid 2\leq i\leq n\}) = \{0,D\}$ and $\{0,D\}$ is the
universe of a subalgebra of $\A(\T)$. The second part follows similarly.

For items $(2)$ and $(3)$, observe that only fundamental operations of
$\A(\T)$ that are nonzero on $B_n$ are $\wedge$, $J$, $J'$, and $S_2$. For
these operations, we have the following inequalities
\begin{align*}
  & u\wedge v \leq u, & 
    & J(u,v,w) \leq u, \\
  & J'(u,v,w) \leq u, &
    & S_2(a,b,u,v,w) \leq u.
\end{align*}
Since items $(2)$ and $(3)$ are true for the generators of $B_n$ and $B_n$
is height $1$ in each coordinate, these inequalities force items $(2)$ and
$(3)$ to also hold for the whole of $B_n$.

The last item, item $(4)$, follows from the previous items.
\end{proof}

We will now proceed to show that $\V(\A(\T))$ does not have definable
principal subcongruences. We begin by defining what it means for an algebra
to have definable principal congruences (DPC) and definable principal
subcongruences (DPSC). A \emph{congruence formula} for a class $\mathcal{C}$
of algebras of the same type is a $4$-ary first order formula
$\psi(w,x,y,z)$ such that for all $\B\in \mathcal{C}$ and all $a,b,c,d\in
B$, if $\B\models \psi(c,d,a,b)$ then $(c,d)\in \Cg^{\B}(a,b)$. If $\psi$ is
such that $\B\models \psi(c,d,a,b)$ if and only if $(c,d)\in \Cg^{\B}(a,b)$,
then we say that $\psi(-,-,a,b)$ \emph{defines} $\Cg^{\B}(a,b)$.

\begin{defn} \label{defn:DPC}
A class $\mathcal{C}$ of algebras of the same type is said to have
\emph{definable principal congruences (DPC)} if there is a congruence formula
$\psi$ such that for every $\B\in \mathcal{C}$ and every $a,b\in B$,
$\psi(-,-,a,b)$ defines $\Cg^{\B}(a,b)$.
\end{defn}

Although the DPC property is quite useful, it is somewhat uncommon. A
weakening of definable principal congruences, called definable principal
subcongruences and introduced in \cite{BakerWangDPSC}, turns out to be much
more common, and still has many of the features that make DPC appealing.

\begin{defn} \label{defn:DPSC}
A class $\mathcal{C}$ of algebras of the same type is said to have
\emph{definable principal subcongruences (DPSC)} if there are congruence
formulas $\Gamma$ and $\psi$ such that for every $\B\in \mathcal{C}$ and
every $a,b\in B$ with $a\neq b$, there exists $c,d\in B$ with $c\neq d$ and
such that $\B\models \Gamma(c,d,a,b)$ and $\psi(-,-,c,d)$ defines
$\Cg^{\B}(c,d)$.
\end{defn}

\newcolumntype{A}{>{\centering\arraybackslash} m{.45\linewidth}}
\newcolumntype{B}{>{\centering\arraybackslash} m{.45\linewidth}}
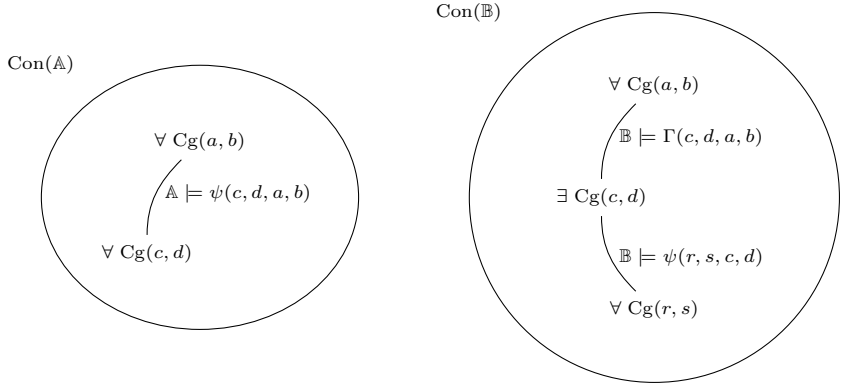
\begin{figure}[h] \begin{tabular}{AB}
  \begin{tikzpicture}[font=\scriptsize,smooth]
    \node (ab) {$\forall \; \Cg(a,b)$};
    \node (cd) [below=of ab,xshift=-2em] {$\forall \; \Cg(c,d)$};
    \coordinate (t1) at (ab|-cd);
    \coordinate (t1) at ($(t1) !0.5! (ab)$);
    \draw (t1) ellipse (6em and 5em)
      (ab) to[out=225,in=90] node[right]{$\A\models\psi(c,d,a,b)$} (cd);
    \node at (current bounding box.north west) {$\Con(\A)$};
  \end{tikzpicture}
  &
  \begin{tikzpicture}[font=\scriptsize,smooth]
    \node (ab) {$\forall \; \Cg(a,b)$};
    \node (cd) [below=of ab,xshift=-2em] {$\exists \; \Cg(c,d)$};
    \node (rs) [below=of cd,xshift=2em] {$\forall \; \Cg(r,s)$};
    \coordinate (t1) at ($(ab) !0.5! (rs)$);
    \draw (t1) ellipse (7em and 7em)
      (ab) to[out=225,in=90] node[right]{$\B\models\Gamma(c,d,a,b)$} (cd)
      (cd) to[out=270,in=135] node[right]{$\B\models\psi(r,s,c,d)$} (rs);
    \node at (current bounding box.north west) {$\Con(\B)$};
  \end{tikzpicture}
\end{tabular}
\caption{$\A$ has DPC via $\psi$, and $\B$ has DPSC via $\Gamma$ and
  $\psi$.}
\end{figure}

If $\mathcal{C}$ is a class with DPSC and $\B\in \mathcal{C}$, then every
nontrivial principal congruence of $\B$ must have a nontrivial
subcongruence that is defined by a fixed congruence formula. Observe that if
the principal congruence in question is atomic, then it is necessarily
definable since it has no proper nontrivial subcongruences. Thus, there is a
single fixed congruence formula that defines every atomic congruence of
every algebra in $\mathcal{C}$.

To show that the subclass consisting of all of the $\B_n$ algebras defined
above does not have DPSC, we will produce an atomic congruence of $\B_n$ for
each $n$, and show that there can be no congruence formula that defines all
of them when $n$ is sufficiently large.

\begin{lem} \label{lem:(a,0)_atomic}
$\Cg^{\B_n}(a,0)$ is an atomic congruence of $\B_n$.
\end{lem}
\begin{proof}
Suppose that $u\neq v$ and $(u,v)\in \Cg^{\B_n}(a,0)$. From Lemma
\ref{lem:B_n_properties}, we have that $\pi_1(B_n) = \{0,D\}$. Since
$a(1)\neq 0$ and $a(i) = 0$ for $i\geq 2$, it follows that $\{u(1),v(1)\} =
\{0,D\}$, so $\{u\wedge a, v\wedge a\} = \{0, a\}$, and thus
$\Cg^{\B_n}(a,0) \subseteq \Cg^{\B_n}(u,v)$. Therefore $\Cg^{\B_n}(a,0)$ is
atomic, as claimed.
\end{proof}

\begin{lem} \label{lem:(bn,cn)_in_(a,0)}
$(b_n,c_n)\in \Cg^{\B_n}(a,0)$.
\end{lem}
\begin{proof}
We have
\begin{align*}
  J'(b_2,d_2,a) = b_2, && && J'(b_2,d_2,0) = c_2, \\
  J'(b_l,d_l,b_{l-1}) = b_l, && && J'(b_l,d_l,c_{l-1}) = c_l
\end{align*}
for all $2\leq l\leq n$. The conclusion follows immediately.
\end{proof}

\begin{lem} \label{lem:f(a)=b_f(0)=c}
If $\m{C}\leq \B_n$, $a,b_n,c_n\in C$ and $f(x)$ is a polynomial of
$\m{C}$ such that $f(a) = b_n \neq f(0)$, then $f(0) = c_n$.
\end{lem}
\begin{proof}
We have that $a(l) = 0$ for all $l\geq 2$. Therefore for all $l\geq 2$,
\[
  D = b_n(l) = f(a)(l) = f(a(l))=f(0(l)) = f(0)(l),
\]
so $f(0)(l) = D$ for all $l\geq 2$. Since $f(0)\neq b_n$, by Lemma
\ref{lem:B_n_properties} it must be that $f(0)(1) = 0$. Thus $f(0) = c_n$.
\end{proof}

The next lemma makes use of the fact that $J$ and $J'$ are $0$-absorbing in
their first and second variables. An operation $F(x_1,\ldots x_n)$ is said
to be \emph{$0$-absorbing in the m-th variable} if 
\[
  F(x_1,\ldots,\stackrel{m}{\hat{0}},\ldots, x_n) \approx 0
\]
holds.

\begin{lem} \label{lem:(bn,cn)_not_in_(a,0)}
$(b_n,c_n)\not\in\Cg^{\m{C}}(a,0)$ for any $\m{C}\lneq \B_n$.
\end{lem}
\begin{proof}
We will use the notation $[i,j]$ to mean the set of those $l\in \ZZ$ such
that $i\leq l\leq j$.  If $\m{C}\lneq \B_n$, then $C$ must omit some of the
generators of $\B_n$. The only generators of $\B$ that $\m{C}$ could
possibly omit are of the form $b_i$ and $d_k$ for some $i\neq n$ and any
$k$. Since $J(b_n,d_k,b_n) = b_k$, if $b_k\not\in C$ then, $d_k\not\in C$.
Thus, we need only consider the case when $d_k\not\in C$. We will show that
if $f(x)$ is a polynomial of $\m{C}$ and $f(a)\neq f(0)$ then there is some
$l\in [1,n]$ such that $f(a)(l) = 0$. The proof shall be by induction on the
complexity of $f(x)$. For $f(x) = x$, the claim clearly holds. Assume now
that the claim holds for all polynomials of complexity less than $f(x)$.  If
$f(x)$ is the result of applying $S_2$ to other polynomials, then by Lemma
\ref{lem:B_n_properties} part $(1)$ and the definition of $S_2$, $f(a)(1) =
f(0)(1) = 0$, so $f(a) = f(0)$. The case where $f(x)$ is the result of the
application of $\wedge$ to two polynomials is also straightforward.

Suppose that $f(x) = J(g_1(x),g_2(x),g_3(x))$. If $g_1(a)\neq g_1(0)$ or
$g_2(a)\neq g_2(0)$, then by the inductive hypothesis $g_1(a)(l) = g_1(0)(l)
= 0$ for some $l\in [1,n]$ or $g_2(a)(l) = g_2(0)(l) = 0$ for some $l\in
[1,n]$. Since $J$ is $0$-absorbing in its first and second variables, this
implies that $f(a)(l) = f(0)(l) = 0$ for some $l\in[1,n]$, as desired.
Assume now that $g_1(a) = g_1(0) = \alpha$ and $g_2(a) = g_2(0) = \beta$.
Then $f(a)\neq f(0)$ implies $\alpha(1) = \partial \beta(1)$, by the
definition of $J$ and $a$.  This contradicts Lemma \ref{lem:B_n_properties}
part $(1)$.

Suppose now that $f(x) = J'(g_1(x),g_2(x),g_3(x))$. If $g_1(a)\neq g_1(0)$
or $g_2(a)\neq g_2(0)$, then by the inductive hypothesis $g_1(a)(l) =
g_1(0)(l) = 0$ for some $l\in [1,n]$ or $g_2(a)(l) = g_2(0)(l) = 0$ for some
$l\in [1,n]$.  Since $J'$ is $0$-absorbing in its first and second
variables, this implies that $f(a)(l) = f(0)(l) = 0$ for some $l\in [1,n]$,
as desired. Assume now that $g_1(a) = g_1(0) = \alpha$ and $g_2(a) = g_2(0)
= \beta$. If $\alpha(l) = 0$ or $\beta(l) = 0$, then $f(a)(l) = 0$, so
assume that $\alpha$ and $\beta$ are nowhere $0$. If $f(a)(l) = 0$, then the
conclusion of the polynomial induction clearly holds, so also assume that
$f(a)$ is nowhere $0$. By Lemma \ref{lem:B_n_properties}, this implies that
$\alpha,\beta,f(a)\in \{b_n,d_n\}$. If $f(a)\neq f(0)$, then it must be that
$g_3(a)\neq g_3(0)$, so by the inductive hypothesis there is some $l\in
[1,n]$ such that $g_3(a)(l) = 0$. From the definition of $J'$, it must
therefore be that $\alpha(l) = \partial \beta(l)$, and since
$\alpha,\beta\in \{b_n,d_n\}$, from the definition of $b_n$ and $d_n$ we
have $l = n$. At this point, if $k = n$ (i.e. $d_n\not\in C$), then we would
have a contradiction, so it must be that $k < n$. It follows then that
$g_3(a) \in \{b_{n-1}, d_{n-1} \}$. Applying the exact same argument as
above to $g_3(x)$, replacing $l\in [1,n]$ with $l\in [1,n-1]$, we conclude
that $k < n-1$. Continuing in this manner, we see that there can be no $k$
such that $d_k\not\in C$, which contradicts our original assumption that $C$
omits some generator of $\B$.

This completes the induction on the complexity of polynomials, so we now
have that if $f(x)$ is a polynomial such that $f(a)\neq f(0)$, then there is
some $l$ such that $f(a)(l) = 0$. In particular, since $b_n\in B_n$ is
nowhere $0$, this means that the congruence class of $b_n$ is trivial, and
cannot contain $c_n$.
\end{proof}

\begin{thm} \label{thm:no_dpsc}
$\V(\A(\T))$ does not have DPSC.
\end{thm}
\begin{proof}
If $\V(\A(\T))$ did have DPSC, then there would be a congruence formula
$\psi(w,x,y,z)$ such that for any algebra in $\V(\A(\T))$, $\psi$ defines
every atomic congruence of that algebra. Since $\B_n \in \V(\A(\T))$ for all
$n$, in particular by Lemma \ref{lem:(a,0)_atomic}, this means that the
congruence $\Cg^{\B_n}(a,0)$ is definable. By Lemma
\ref{lem:(bn,cn)_in_(a,0)}, $(b_n,c_n)\in \Cg^{\B_n}(a,0)$. Therefore there
is some number $N$ (depending only on $\V(\A(\T))$) such that $(b_n,c_n)\in
\Cg^{\B_n}(a,0)$ implies $(b_n,c_n)\in \Cg^{\m{C}}(a,0)$ for some subalgebra
$\m{C}$ of $\B_n$ with at most $N$ generators. Lemma
\ref{lem:(bn,cn)_not_in_(a,0)} states, however, any such $\m{C}$ must
actually be equal to $\B_n$, and since the minimum number of generators of
$\B_n$ goes to infinity as $n$ does, this is a contradiction. Thus
$\V(\A(\T))$ cannot have DPSC.
\end{proof}

In the algebra $\A'(\T)$ from \cite{MooreDPSCUndec}, the $K$ operation could
be used to produce an element $b_n'$ such that $b_n(i) = \partial b_n'(i)$
for $i \geq 2$:
\begin{align*}
  b_2' = d_n, 
  && &&
  b_{k+1}' = K(b_n, b_k', d_{n-(k-1)}).
\end{align*}
This $b_n'$ can then be used to witness $(b_n,c_n)\in \Cg^{\B_n}(a,0)$ via
the polynomial $\lambda(x) = J'(b_n,b_n',x)$. That is, $\lambda(a) = b_n$
and $\lambda(0) = c_n$. The element $b_n'$ is also a counterexample to Lemma
\ref{lem:B_n_properties}, which is used heavily in the above proofs.
$\V(\A(\T))$ fails to have DPSC for any $\T$, but the addition of the $K$
operation links DPSC to the halting status of the Turing machine $\T$.

Next, we prove that $\V(\A(\T))$ does not have bounded Maltsev depth. From
Maltsev's description of principal congruences, we have $(c,d)\in \Cg(a,b)$
if and only if there are elements $c = r_1, r_2, \ldots, r_n = d$ and unary
polynomials $\lambda_1(x), \ldots, \lambda_{n-1}(x)$ such that
$\{\lambda_i(a), \lambda_i(b) \} = \{r_i, r_{i+1}\}$. The property of
bounded Maltsev depth (introduced in \cite{BakerWangApproxDistLaws}) is
motivated by the observation that in a congruence distributive variety
generated by a finite algebra, there is a bound $M$ such that it is
sufficient in the above description of principal congruences to only
consider those $\lambda_i(x)$ can all be taken to be compositionally
generated by at most $M$ fundamental translations (a fundamental translation
is a unary polynomial that is the result of fixing all but one variable in a
fundamental operation).

\begin{defn} \label{defn:bounded_Maltsev_depth}
Let $M$ be a natural number. A class $\mathcal{C}$ of algebras of the same
type is said to have \emph{Maltsev depth $M$} if for every $\A\in
\mathcal{C}$ and every $a,b,c,d\in A$ such that $(c,d)\in \Cg^{\A}(a,b)$
there are elements $c = r_1, r_2, \ldots, r_n = d$ and unary polynomials
$\lambda_1(x), \ldots \lambda_{n-1}(x)$ such that
\[
  \{ \lambda_i(a), \lambda_i(b) \} = \{ r_i, r_{i+1} \}
\]
and each $\lambda_i(x)$ is compositionally generated by at most $M$
fundamental translations, and $M$ is minimal with this property. 

The class $\mathcal{C}$ is said to be of \emph{bounded Maltsev depth} if
there is some $M$ such that $\mathcal{C}$ has Maltsev depth $M$.
\end{defn}

In the next lemma, the \emph{support} of $\alpha\in \A(\T)^n$ is
$\supp(\alpha) = \{ l\in [1,n] \mid \alpha(l)\neq 0 \}$.

\begin{lem} \label{lem:produce_nonzeros}
Suppose that $r,s\in B_n$ are such that $r(1) \neq s(1) = 0$ and $r(i) =
s(i)$ for $i\geq 2$. If $g(x)$ is a fundamental translation such that
$g(r)\neq g(s)$, then $|\supp(g(r))| \leq |\supp(r)| + 1$.
\end{lem}
\begin{proof}
The proof of this lemma is somewhat similar to the proof of Lemma
\ref{lem:(bn,cn)_not_in_(a,0)}. If $g(x)$ is a translation of
$S_2$, then from Lemma \ref{lem:B_n_properties} part $(1)$, the
definition of $S_2$, and the hypotheses concerning $r$ and $s$, we have that
$g(r) = g(s)$. If $g(x)$ is a translation of $\wedge$, then certainly
$|\supp(g(r))| \leq |\supp(r)|$.

If
\[
  g(x) \in \{ J(x, \alpha, \beta), J(\alpha, x,\beta), J'(x, \alpha, \beta),
  J'(\alpha, x, \beta) \},
\]
then since $J$ and $J'$ are $0$-absorbing in their first and second
variables, $|\supp(g(r))| \leq |\supp(r)|$. If $g(x) = J(\alpha,\beta,x)$,
then from the hypotheses concerning $r$ and $s$, and the definition of $J$,
$g(r)\neq g(s)$ implies $\alpha(1) = \partial \beta(1)$, contradicting Lemma
\ref{lem:B_n_properties} part $(1)$.

The last remaining case is $g(x) = J'(\alpha,\beta,x)$. Let $r' = g(r)$ and
suppose that there are $k,l$ such that $r(k) = r(l) = 0$ but $r'(k) \neq 0
\neq r'(l)$. From the definition of $J'$, this implies that $\alpha(k) =
\partial \beta(k)$ and $\alpha(l) = \partial \beta(l)$. Lemma
\ref{lem:B_n_properties} parts $(2)$ and $(3)$ then imply that $k = l$, so
it follows that $|\supp(g(r))| \leq |\supp(r)| + 1$.
\end{proof}

\begin{thm} \label{thm:not_bounded_Matlsev}
$\V(\A(\T))$ does not have bounded Maltsev depth.
\end{thm}
\begin{proof}
Since $(b_n,c_n)\in \Cg^{\B_n}(a,0)$, there is some polynomial $f(x)$
generated by fundamental translations such that $f(a) \neq f(0)$ and $f(a) =
b_n$. Say that $f(x) = f_m(f_{m-1}(\cdots f_1(x) \cdots ) )$ for fundamental
translations $f_i(x)$.

By applying Lemma \ref{lem:produce_nonzeros} with $g(x) = f_i(x)$, $r =
f_{i-1}(f_{i-2}(\cdots f_1(a) \cdots ) )$, and $s = f_{i-1}(f_{i-2}(\cdots
f_1(0)\cdots ) )$ for each $i$, we have that $m \geq n-1$. Hence $f(x)$ has
nesting depth at least $n-1$. Therefore $\B_n$ has Maltsev depth of at least
$n-1$, and since $\B_n\in \V(\A(\T))$ for all $n\in \ZZ_{\geq 2}$, it
follows that $\V(\A(\T))$ does not have bounded Maltsev depth.
\end{proof}

Recall from the introduction that our goal in this paper has been to provide
negative answers to the following questions posed by Willard:
\begin{enumerate}
  \item if $\A$ is finite of finite type and $\V(\A)$ has finite residual
  bound and is congruence $\wedge$-semidistributive, is it true that
  $\V(\A)$ has definable principal subcongruences?

  \item if $\V$ is a congruence $\wedge$-semidistributive variety in a
  finite language and has finite residual bound, is it true that $\V$ has
  bounded Maltsev depth? 
\end{enumerate}
These questions are answered in Theorems \ref{thm:no_dpsc} and
\ref{thm:not_bounded_Matlsev}. Negative answers to these questions means
that neither DPSC nor bounded Maltsev depth will lead to a simplification of
Willard's Finite Basis Theorem, as was the case for Baker's Finite Basis
Theorem (see \cite{BakerWangDPSC} for DPSC and
\cite{BakerWangApproxDistLaws} for bounded Maltsev depth).

As mentioned in the introduction, the question of whether the algebra
$\A'(\T)$ has bounded Maltsev depth when $\T$ halts is unanswered, but an
approach involving a careful analysis of polynomials of $\A'(\T)$ would seem
to be necessary. A similar analysis showed that $\A'(\T)$ has DPSC if $\T$
halts, and it may be that the analysis for bounded Maltsev depth can build
on this without too much additional work.

\bibliographystyle{amsplain}
\bibliography{algebra-without-dpsc_references}
\begin{center}
  \rule{0.61803\textwidth}{0.1ex}   
\end{center}
\end{document}